\documentclass[12pt]{amsart}

\usepackage{graphicx}              
\usepackage{amsmath}               
\usepackage{amsfonts}              
\usepackage{amsthm}                
\usepackage{mathrsfs}
\usepackage{tikz-cd}
\usetikzlibrary{matrix,arrows,decorations.pathmorphing}
\usepackage[italian,english]{babel}
\usepackage[utf8]{inputenc}
\usepackage{MnSymbol}
\usepackage[autostyle]{csquotes}
\usepackage[margin=1in]{geometry}

\newtheorem{thm}{Theorem}[section]
\newtheorem{lem}[thm]{Lemma}
\newtheorem{prop}[thm]{Proposition}
\newtheorem{cor}[thm]{Corollary}
\newtheorem{conj}[thm]{Conjecture}

\theoremstyle{remark}

\newtheorem{rmk}[thm]{Remark}

\theoremstyle{definition}
\newtheorem{defin}[thm]{Definition}
\newtheorem{definlem}[thm]{Definition-Lemma}
\newtheorem{hyp}[thm]{Hypothesis}

\newcommand{\ZZ}{\mathcal{Z}}
\newcommand{\CC}{\mathbb{C}} 
\newcommand{\PP}{\mathbb{P}}
\newcommand{\QQ}{\mathbb{Q}} 
\newcommand{\OO}{\mathcal{O}}
\newcommand{\LL}{\mathscr{L}}

\newcommand{\MM}{\overline{\mathcal{M}}}

\newcommand{\proj}{\operatorname{Proj}}

\begin{document}
\date{}
\nocite{*}

\title{\textbf{On the arithmetic of weighted complete intersections of low degree}}

\author {cristian minoccheri}
\address{Department of Mathematics, University of Michigan, Ann Arbor, MI 48109}
\email{minoc@umich.edu}

\begin{abstract}
A variety is rationally connected if two general points can be joined by a rational curve. A higher version of this notion is rational simple connectedness, which requires suitable spaces of rational curves through two points to be rationally connected themselves. We prove that smooth, complex, weighted complete intersections of low enough degree are rationally simply connected. This result has strong arithmetic implications for weighted complete intersections defined over the function field of a smooth, complex curve. Namely, it implies that these varieties satisfy weak approximation at all places, that $R$-equivalence of rational points is trivial, and that the Chow group of zero cycles of degree zero is zero.
\end{abstract}

\maketitle

\tableofcontents

\section{Introduction}

For a variety defined over a field $F$, a basic question is whether it has points defined over $F$, and, if so, what properties do these points satisfy. One such property that has been extensively studied with different techniques over various fields is the weak approximation principle: for any finite set of places of $F$ and points of $X$ over these places, is there an $F$-rational point of $X$ which approximates these points arbitrarily closely? If $F$ is the function field $K=\CC(B)$ of a smooth complex curve $B$, this property has a nice geometric interpretation: for a (regular) model $\pi: \mathscr{X} \to B$ of $X$, and for any choice of jets over finitely many points of $B$, is there a section of $\pi$ with these prescribed jets?

\

Rationally connected varieties -- varieties whose any two points can be connected by a rational curve -- are natural candidates for the weak approximation principle, since by \cite{GHS} any geometrically rationally connected variety over $K$ admits a $K$-rational point. In \cite{HT}, Hassett and Tschinkel proved that, in fact, if $X$ is geometrically rationally connected over the function field $K$ of a curve, the weak approximation principle holds at places of good reduction (i.e., at points of $B$ that admit smooth fibers in some model). It is then natural to wonder whether weak approximation holds at places of bad reduction too, and indeed there is an open conjecture (\cite[Conjecture 2]{HT}) that reads as follows:

\begin{conj}
A smooth geometrically rationally connected variety $X$ over $K$ satisfies weak approximation at all places.
\end{conj}

This conjecture has proven hard to tackle, even though by now many cases are known (see \cite{Has} for a survey of the progress in this direction). One way to attack it is to prove that $X$ is \em rationally simply connected\em. This notion is analogous to that of simple connectedness in topology (we will provide a precise definition in section $2$) and in a nutshell, it pertains the study of rational connectedness of suitable spaces of stable maps to $X$ of genus $0$. By a Theorem of Hassett (\cite[Theorem 4.7]{Has}), if $X$ is geometrically rationally simply connected, then $X$ satisfies weak approximation at all places. This notion has further arithmetic consequences, due to work of Pirutka (\cite{Pir}): if $X$ is geometrically rationally simply connected, there is only one class of Manin's $R$-equivalence, and the degree map $deg:CH_0(X) \to \mathbb{Z}$ on zero cycles is bijective.

Despite its many implications, the main problem with the notion of rational simple connectedness is that it is quite hard to check for concrete examples. For $X$ to be rationally simply connected, one usually needs the tangent bundle of $X$ to satisfy strong positivity properties. In particular, we say that $X$ is $2$-Fano if $c_1(X)$ and $ch_2(X)$ are both positive. In \cite{dJS06a}, de Jong and Starr proved that $2$-Fano complete intersections in projective space are rationally simply connected. While $2$-Fano varieties have only been classified for index at least $n-2$, it seems inevitable for the class of complete intersections in \em weighted projective spaces \em to come up. They are therefore a natural case to study, and in this paper we indeed prove that most $2$-Fano weighted complete intersections are rationally simply connected.

\

Let $\PP_{F}(1^3,e_3,...,e_n)$ denote the weighted projective space over a field $F$ with weights $1,1,1,e_3,...,e_n$ (where $e_3,...,e_n$ are any natural numbers). Our main result is the following:

\begin{thm}
\label{mainthm}
Let $X_{d_1,...,d_c} \subset \PP_{\CC}(1^3,e_3,...,e_n)$ be a smooth, $2$-Fano, weighted complete intersection of dimension at least $3$ which is not isomorphic to a linear space. Assume further that $d_1+...+d_c \leq e_3+...+e_n$. Then $X_{d_1,...,d_c}$ is rationally simply connected.
\end{thm}

As already mentioned, by \cite[Theorem 4.7]{Has} and \cite[Theorem 1.5]{Pir}, we have the following main application to varieties defined over the function field $K$ of a smooth complex curve: 

\begin{cor}
\label{maincor}
Let $X_{d_1,...,d_c} \subset \PP_{K}(1^3,e_3,...,e_n)$ be a smooth, $2$-Fano, weighted complete intersection of dimension at least $3$ which is not isomorphic to a linear space. Assume further that $d_1+...+d_c \leq e_3+...+e_n$. Then:

1) $X_{d_1,...,d_c}$ satisfies weak approximation at all places;

2) $X(K)/R=1$ (i.e., there is only one class of $R$-equivalence); and

3) $deg:CH_0(X_{d_1,...,d_c}) \to \mathbb{Z}$ is bijective.
\end{cor}

Theorem \ref{mainthm} proves rational simple connectedness for most $2$-Fano weighted complete intersections, and in infinitely many new cases. The condition requiring the first three weights to be $1$ shouldn't be restrictive: in all the explicit examples we have, for a weighted, $2$-Fano complete intersection to be smooth, at least $3$ weights have to be $1$ (in fact, often many more than $3$). One could expect the result to still hold if we replaced the bound $d_1+...+d_c \leq e_3+...+e_n$ with the sharp Fano bound $d_1+...+d_c \leq e_3+...+e_n+2$. On the other hand, our method requires the weighted complete intersection to have a positive dimensional family of lines through a general point; therefore the condition on the index is necessary.

The easiest case in which the Theorem applies is that of cyclic degree $r$ covers of $\PP^{n-1}$ branched along a hypersurface of degree $r \cdot e$. Namely, consider the case $e_3=...=e_{n-1}=1$, $e:=e_n \geq 1$, $r \geq 1$, $X_{re} \subset \PP_K(1,...,1,e)$ smooth weighted hypersurface of degree $re$. Then $X_{re}$ satisfies the properties of Corollary \ref{maincor} if $n \geq max\{4, re-e+3, r^2e^2-e^2\}$. In the case of double covers, i.e. $r=2$, and for $e \geq 2$, the condition becomes $n \geq 3e^2$.

\

In \cite{dJS06a}, a general strategy to prove rational simple connectedness is designed (in fact, the authors actually prove a \em stronger \em version of rational simple connectedness than the one we need in this paper for the arithmetic implications we mentioned above). We will follow such general framework. On the other hand, weighted projective spaces are both very close and very far from usual projective space in several respects; therefore, while many steps of the proof will adapt with little effort, several key steps will require substantially different proofs. The paper is organized as follows.

In section $2$, we collect basic facts about weighted projective spaces, spaces of stable maps, and rational simple connectedness that we will need in the following sections. We also prove in the weighted case the analogue of a technique that allows to extend a smooth complete intersection to a smooth complete intersection of higher dimension, of which the former is a linear section.

In section $3$, we prove that the generic geometric fiber of $ev_1: \MM_{0,1}(X,1) \to X$ is connected. The proof of this in \cite{dJS06a} relies on the very basic fact that there is only one line through two points; but this fails in the weighted case. Our proof consists in describing the fiber of the evaluation map as a closed subscheme of an appropriate weighted projective space, whose coordinates parametrize suitable morphisms from $\PP^1$ into the original weighted projective space (this is sometimes called the space of quasi-maps). We then make use of a connectedness result that is a generalized, iterated version of the classical connectedness Lemma of Enriques-Severi-Zariski.

In section $4$, we prove that the generic geometric fiber of $ev_2: \MM_{0,2}(X,2) \to X \times X$ is connected. Again, the method in \cite{dJS06a} does not immediately apply for essentially the same reason as above. We first find an ample divisor class $\lambda$ on the fiber, and relate it to other natural classes. Then we show the result for higher dimensional weighted complete intersections by means of a weighted version (due to B\u adescu) of the Fulton-Hansen Connectedness Theorem. Finally, we show how the original fibers are obtained as sections of the higher dimensional ones by divisors with class $\lambda$. In the process, we also show that the general fiber of $ev_2$ is uniruled by rational curves with $\lambda$-degree $1$.

In section $5$, we prove existence of some special ruled surfaces in $X$ called $1$-twisting surfaces. This will be a by-product of uniruledness of a general fiber of $ev_2$ proven in section $4$. From this, one can deduce rational simple connectedness from the case of conics, as in \cite{dJS06a}.

\

{\bf Acknowledgments:} I would like to thank my advisor, Jason Starr, for proposing the problem, for many invaluable suggestions, and for his constant support and encouragement. I would also like to thank Luigi Lombardi and David Stapleton for many useful conversations.

\section{Preliminaries}

We collect here some general definitions and results that are going to be used later. Let us first start with some precise definitions of the properties we are going to study.

\subsection{Arithmetic notions}

\

\textbf{Weak approximation over function fields}. Let $B$ be a smooth, connected curve over an algebraically closed field (say, $\CC$). Let $K$ be the function field of $B$, and $S \subset B$ a proper, closed subscheme. Let $X$ be a smooth $K-$scheme, and $\mathscr{X}$ a (regular) \em model \em for $X$, i.e., a proper morphism $\pi: \mathscr{X} \to B$ with regular total space $\mathscr{X}$ and with generic fiber $X$. Then $X$ satisfies \em weak approximation with respect to $B$ and $S$ \em if for one model $\mathscr{X}$ and every $B-$morphism $\sigma: S \to \mathscr{X}$, $\sigma$ can be extended to a $B-$morphism defined on an open subset $U \subset B$ that contains $S$. $X$ satisfies \em weak approximation with respect to $B$ \em (or, equivalently, \em over $K$\em) if for every proper, closed subscheme $S \subset B$, $X$ satisfies weak approximation with respect to $B$ and $S$.

Koll\'ar, Miyaoka and Mori showed in \cite{KMM} that if $X$ is rationally connected and a model admits a section, then this section enjoys good approximation properties. Years later, Graber, Harris and Starr proved in \cite{GHS} the existence of such sections. Finally, Hassett and Tschinkel proved in \cite{HT} that weak approximation holds at places of good reduction (i.e. at points $b \in B$ for which there is a model having a smooth fiber over $b$). They also conjectured that every rationally connected variety over a function field $K$ as above satisfies weak approximation. After that, several authors managed to prove weak approximation in many cases by studying explicitly singular fibers. On the other hand, by \cite[Theorem 4.7]{Has}, if $X$ is geometrically rationally simply connected, it satisfies weak approximation at all places, regardless of the singularities that a model can have.

\

\textbf{$R$-equivalence}. Given a projective variety $X$ over a field $K$, we say that two $K-$points $x,y \in X(K)$ are \em directly R-equivalent \em if there is a morphism $\PP^1_K \to X$ such that $x$ and $y$ belong to the image of $\PP^1_K(K)$. One obtains an equivalence relation (due to Manin) called \em $R$-equivalence\em, whose set of classes is denoted $X(K)/R$. In \cite{Pir}, Pirutka proves that if $X$ is a rationally simply connected variety over the function field $K$ of a smooth complex curve, then $X(K)/R=1$. Moreover, she proves that with the same setup, the Chow group of zero cycles of degree zero is zero.

\subsection{Weighted projective spaces and complete intersections}

We now review some useful properties of weighted projective spaces and weighted complete intersections. We refer the reader to \cite{Dol} and \cite{Mor} for proofs and further discussions.

\

Let $F$ be a field of characteristic $0$. The \em weighted projective space \em $\PP_F:=\PP_F(e_0,...,e_n)$ of dimension $n$ and positive weights $e_0,...,e_n$ can be defined as $\proj F[x_0,...,x_n]$, where $F[x_0,...,x_n]$ is the polynomial ring with variables $x_i$ of degree $e_i$ for $i=0,...,n.$

It is customary to write iterated weights as powers of the weight; for example, $\PP(1^3,2,3^2)$ will be used to denote $\PP(1,1,1,2,3,3)$.

In this paper, $F$ will be either the field of complex numbers (in which case the subscript in $\PP_F$ will be omitted), or the function field $K$ of a smooth, connected, complex curve. We can always assume that any $n$ weights are coprime. Weighted projective spaces behave in many ways like usual projective space, and in many other ways differently, as we will see.

\

The weighted projective space $\PP$ is an irreducible, normal, projective variety, with cyclic quotient singularities (unless all weights are equal to $1$). Let $V_k$ be the closed subset of $\PP$ defined by the ideal $< x_i \ | \ k \nmid e_i>$. Then the singular locus $Sing(\PP)$ of $\PP$ equals $\cup_{k > 1} V_k$, and $\PP^\circ:=\PP \setminus Sing(\PP)$ is called \em weak projective space\em.

The weighted projective space $\PP$ has natural coherent $\OO_\PP-$modules $\OO_{\PP}(a)$ for every integer $a$, associated via the $\proj$ construction to the modules $F[x_0,...x_n](a)$. In general, the sheaves $\OO_{\PP}(a)$ are neither invertible, nor ample, and they don't behave well under tensor product. However, if we define $\OO_{\PP^\circ}(a)$ as $\OO_{\PP}(a)|_{\PP^\circ}$, the sheaves $\OO_{\PP^\circ}(a)$ are invertible, they behave well with respect to tensoring, and they are ample when $a$ is positive.

\

While working with differentials on $\PP$ requires some extra care, everything works as we expect once we restrict to the weak projective space. In particular, we have the following generalization of Euler's sequence: $$0 \to \OO_{\PP^\circ} \to \oplus_{i=0}^n \OO_{\PP^\circ}(e_i) \to T_{\PP^\circ} \to 0.$$

We will now discuss weighted complete intersections, whose definition is the standard one. For $f_1,...,f_c$ homogeneous elements of $\proj \CC[x_0,...,x_n]$ of degrees $d_1,...,d_c$, that form a regular sequence, $\proj \CC[x_0,...,x_n]/(f_1,...,f_c)$ is a \em complete intersection \em in $\PP$ (also called \em weighted complete intersection\em) of degrees $d_1,...,d_c$ (and codimension $c$). Denote such a weighted complete intersection by $X_{d_1,...,d_c} \subset \PP$.

\

We will be interested exclusively in smooth weighted complete intersections. We say that a weighted complete intersection $X_{d_1,...,d_c} \subset \PP$ is smooth if it is a smooth variety (in the usual sense) \em and \em it is contained in $\PP^\circ$. This definition is standard (see \cite{Mor} for example), and it is necessary for smooth weighted complete intersections to share a few nice properties of complete intersections in usual projective space. What could happen is that a variety $X \subset \PP$ is smooth despite having nonempty intersection with $Sing(\PP)$. On the other hand, it can be proved that if a weighted complete intersection $X=X_{d_1,...,d_c} \subset \PP$ is a smooth variety (in the usual sense) \em and \em $codim_X(X \cap Sing(\PP)) \geq 2$, then $X \subset \PP^\circ$.

\

If $X:=X_{d_1,...,d_c} \subset \PP$ is a smooth weighted complete intersection, then the sheaf $\OO_X(1):=\OO_{\PP^\circ}(1)|_{X}$ is ample and invertible, and we have the following formula for the canonical sheaf: $$\omega_X \simeq \OO_X(\sum_{j=1}^c d_j - \sum_{i=0}^n e_i).$$

All degrees will be intended with respect to $\OO_X(1)$, unless otherwise specified.

\

With respect to cohomology, the weighted case is also as good as possible. In fact, we have the following result (see \cite[Appendix B, Theorem B.13]{Dim} and \cite[Appendix B, Theorem B.22]{Dim}):

\begin{thm}
\label{cohomology}
The rational cohomology algebra $H^*(\PP,\QQ)$ is a truncated polynomial algebra $\QQ[z]/(z^{n+1})$ generated by an element $z$ of degree $2$.

If $X$ is a weighted complete intersection in $\PP$, then the morphism $H^k(\PP,\QQ) \to H^k(X,\QQ)$ is an isomorphism for $k < dim(X)$, and it is injective for $k=dim(X)$.

\end{thm}

The following computation can be carried on like in usual projective space. We recall that a variety $X$ is called $2$-Fano if it is Fano and $ch_2(T_X)$ is positive; therefore it is useful to have an explicit formula for the Chern character.

\begin{lem}
\label{cherntwo}
Let $X := X_{d_1,..,d_c} \subset \PP(e_0,...,e_n)^{\circ}$ be a smooth weighted complete intersection. Then $$ch_2(T_X)= \frac{1}{2}(\sum e_i^2 -\sum d_j^2) c_1(\OO_X(1))^2.$$
\end{lem}
\begin{proof}
Euler's generalized sequence $$0 \to\OO_{\PP^{\circ}}  \to \oplus_{j=0}^n \OO_{\PP^{\circ}}(e_j) \to  T_{\PP^{\circ}} \to 0$$ gives the relation $$ch(T_{\PP^{\circ}})=\sum_{j=0}^n e^{c_1(\OO_{\PP^{\circ}}(e_j))} -1.$$
The normal exact sequence $$0 \to T_X \to T_{\PP^{\circ}}|_X \to \oplus_{i=1}^c \OO_{\PP^{\circ}}(d_i)|_X \to 0$$ gives the relation $$ch(T_X)=ch(T_{\PP^{\circ}}|_X) - \sum_{i=1}^c e^{d_ic_1(\OO_X(1))},$$ which together with the previous one gives $$ch(T_X)=\sum_{j=0}^n e^{e_jc_1(\OO_X(1))} -1 - \sum_{i=1}^c e^{d_ic_1(\OO_X(1))}.$$ Looking at the part in degree $2$, we get: $$ch_2(T_X)=\frac{1}{2}\sum e_i^2 c_1(\OO_X(1))^2- \frac{1}{2}\sum d_j^2 c_1(\OO_X(1))^2.$$
\end{proof}

We prove now an extendability result which is a key ingredient in the method of \cite{dJS06a}. The proof is different from the original, so that we can apply it to the weighted case.

\begin{prop}
\label{extend}
Let $X:=\ZZ(F_1,...,F_c) \subset \PP:=\PP(e_0,...,e_n)$ be the common zero locus of polynomials $F_1,...,F_c$. Assume that $X$ is smooth, and consider, for a fixed $s>0$, the natural inclusion $\PP \subset \Pi:=\PP(e_0,..,e_n,1^s)$. Then there exists $Y:=\ZZ(G_1,...,G_c) \subset \Pi$ common zero locus of polynomials $G_1,...,G_c$ such that $Y$ is smooth and $Y\cap \PP=X$.

In particular, if $X$ is a smooth weighted complete intersection, $Y$ is also a smooth weighted complete intersection.
\end{prop}
\begin{proof}
Let $R$ and $S$ denote the polynomial rings $\CC[ x_0 , ... , x_n]$ and $\CC[ x_0 , ... , x_n , y_1, ... , y_s ]$ respectively, with $deg(x_i) = e_i$ for every $i=0,...,n$, and $deg(y_1) = ... = deg(y_s) = 1$. Therefore $\PP=Proj(R)$ and $\Pi=Proj(S)$.

\

Denote by $\Pi^\circ$ the open complement of $\PP$ in $\Pi$. Then $\Pi^\circ$ is smooth: namely, for every point $p$ in $\Pi^\circ$, there exists at least one coordinate $y_j$ of $\Pi$ that is nonzero at $p$; therefore $p$ is a smooth point of $\Pi^\circ$.

\

Next we construct the correct parameter space for $c-$uples on $\Pi$ that extend $X$. Let $N$ be the kernel of the algebra homomorphism $S \to R$ that corresponds to the projection $\Pi \to \PP$ onto the first $n+1$ coordinates. Note that, if $S_k$ is the $k-$th graded part of $S$, then $F_i \in S_{e_i}$. For every $i=0,..,n$, denote by $V_i$ the smallest subspace of $S_{e_i}$ that contains $N \cap S_{e_i}$ and $F_i$. The parameter space we seek is the affine space $V := V_1 \times ... \times V_c$ whose points are ordered $c-$uples $( G_1 , ... , G_c )$.

\

Finally, we have to show that we can find a smooth $c-$uple $( G_1 , ... , G_c )$ in $V$. Consider the closed subscheme $E$ of $V \times \Pi^\circ$ parametrizing data $( ( G_1 , ... , G_c ) , p )$ with $G_i(p)=0$ for every $i=1,...,c$.  Since $y_j$ is nonzero at $p$,  $y_j^{e_i} \in V_i$ is also nonzero at $p$. Therefore, the conditions $G_i(p) = 0$ for $i = 1 , ... , c$ give $c$ linearly independent conditions on $V$.  Thus the projection $\pi_2 : E \to \Pi^\circ$ is a vector bundle over $\Pi^\circ$. This implies that $E$ is smooth, since $\Pi^\circ$ is smooth and $\pi_2$ is a smooth morphism.

By generic smoothness, the projection $\pi_1 : E \to V$ is smooth over a dense open subscheme of $V$.  Thus, for $( G_1 , ... , G_c )$ in $V$ general and $Y$ its common zero locus, the open subscheme $Y^\circ := Y \cap \Pi^\circ$ is smooth. Therefore, the only possible singular points of $Y$ are points in the intersection with the zero scheme $\ZZ( y_1 , ... , y_s )$ corresponding to $\PP$. This intersection equals $X$ by construction, and $X$ is smooth by hypothesis, which implies that $Y$ is smooth at every point of $X$. Therefore $Y$ is smooth both at every point of $Y^\circ$ and at every point of $X$, meaning that $Y$ is everywhere smooth.

\

Finally, assume that $X$ is a smooth weighted complete intersection. Then $X$ is an ample effective divisor of $Y \cap \{y_2=...=y_s=0\}$ (since it is cut by $y_1=0$), and therefore $Y \cap \{y_2=...=y_s=0\}$ is a weighted complete intersection by \cite[Corollary 3.8]{Mor}. Iterating this process, we have that $Y$ is a weighted complete intersection.
\end{proof}

For convenience, we collect here the conditions on weights of the projective space and degrees of the complete intersection that will be used throughout the paper:

\begin{hyp}[Main hypothesis]
Let $F$ be a field, and let $X_{d_1,...,d_c} \subset \PP_F(e_0,...,e_n)$ be a weighted complete intersection of degrees $d_1,...,d_c$. We say that $X$ satisfies the \em main hypothesis \em if the following conditions are satisfied:

(1) $X$ has dimension at least $3$, i.e. $c \leq n-3$;

(2) $e_0=e_1=e_2=1$;

(3) $e_3+...+e_n +3+c-n \leq d_1+...+d_c$; 

(4) $d_1+...+d_c \leq e_3+...+e_n$; and

(5) $d_1^2+...+d_c^2 \leq 3+e_3^2 +...+e_n^2$.
\end{hyp}

These inequalities are not as restrictive as it may look at first sight.

We should remark that if $e_3+...+e_n +3+c-n > d_1+...+d_c$, then $X \simeq \PP^{n-c}$ by a characterization of projective space of Cho, Miyaoka and Shepherd-Barron (see \cite{CMSB}). Therefore (3) amounts to asking that $X_{d_1,...,d_c}$ is not isomorphic to usual projective space, a case which is not relevant. On the other hand, this condition implies that there are no curves of degree $1$ through two general points of $X_{d_1,...,d_c}$.

Condition (2) puts a restriction on three of the weights. As already mentioned, this is also probably not restrictive. From the classifications we have of smooth, $2$-Fano weighted complete intersections of high index or low dimension, at least three weights are equal to $1$. In fact, smoothness alone for weighted complete intersections often requires many more weights being equal to $1$. (4) ensures that $X$ is Fano -- precisely, of index at least $3$ -- and furthermore covered by lines (i.e., irreducible, rational curves of degree $1$). Even though the bound is stronger than the sharp Fano bound (i.e., $d_1+...+d_c \leq e_3+...+e_n+2$), all instances of the technique we adopt require such condition on the index. On the other hand, once we require $X$ to be $2-$Fano, this subtlety seems irrelevant from the known examples we have. (5) gives exactly the $2$-Fano bound.

\

In conclusion, it seems to us that this is almost the best possible hypothesis to adopt the techniques in this paper. In a more concise and less technical version, it amounts to the following:

\begin{hyp}[Main hypothesis -- Equivalent formulation] \ We say that $X_{d_1,...,d_c} \subset \PP_F(1^3,e_3,...,e_n)$ satisfies the main hypothesis if it is a smooth, 2-Fano weighted complete intersection of degrees $d_1,...,d_c$, with index and dimension at least $3$, that is not isomorphic to a linear space.\end{hyp}

\subsection{Moduli of stable maps and Rational Simple Connectedness}

An important role in what follows will be played by the moduli space of rational stable maps $\MM_{0,m}(X,\beta)$. We recall here its basic properties, and refer the reader to \cite{FP} and \cite{CK} for further details.

\begin{defin}
Let $X$ be a smooth, projective, complex variety. Let $\beta \in CH_1(X)$ be the class of a curve. An \em $m-$pointed, genus $0$ stable map to $X$ \em is the datum $(C,p_1,...,p_m,f)$ of:

1) a projective, connected, reduced, at worst nodal curve $C$ of arithmetic genus $0$;

2) $m$ distinct smooth marked points $p_1,...,p_m$ on $C$;

3) a morphism $f: C \to X$ satisfying the following stability condition: every contracted component of $C$ via $f$ must contain at least $3$ distinguished points (i.e., nodes or marked points).
\end{defin}

\begin{thm}
Let $X$ be a smooth, projective, complex variety. Let $\beta \in CH_1(X)$ be the class of a curve. Then there exists a projective, coarse moduli scheme $\MM_{0,m}(X,\beta)$ over $\CC$ parametrizing isomorphism classes $[(C,p_1,...,p_m,f)]$ of $m-$pointed, genus $0$ stable maps to $X$ such that $f_*([C])=\beta$.
\end{thm}

In general, $\MM_{0,m}(X,\beta)$ is neither smooth nor irreducible, and it might have irreducible components of different dimensions. In any event, the following Theorem (see \cite{CK}, 7.1.4) describes the local structure of $\MM_{0,m}(X,\beta)$. Here, $f^* \Omega^1_X \to \Omega^1_C(\sum_{i=1}^m p_i)$ is a complex in degrees $-1$ and $0$, and $Ext^k$ for $k=1,2$ denotes the hyperext groups.

\begin{thm}

Let $[(C,p_1,...,p_m,f)] \in \MM_{0,m}(X,\beta)$ be a stable map. Then $\MM_{0,m}(X,\beta)$ is defined locally around $[(C,p_1,...,p_m,f)]$ by $$dim\ Ext^2(f^* \Omega^1_X \to \Omega^1_C(\sum_{i=1}^m p_i), \OO_C)$$ equations in a nonsingular scheme of dimension $$dim\ Ext^1(f^* \Omega^1_X \to \Omega^1_C(\sum_{i=1}^m p_i), \OO_C).$$ In particular, every irreducible component of $\MM_{0,m}(X,\beta)$ has dimension at least $$dim\ Ext^2(f^* \Omega^1_X \to \Omega^1_C(\sum_{i=1}^m p_i), \OO_C)-dim\ Ext^1(f^* \Omega^1_X \to \Omega^1_C(\sum_{i=1}^m p_i), \OO_C)=$$ $$=-K_X \cdot \beta + dim(X) +m -3.$$

\end{thm}

The smallest possible dimension $-K_X \cdot \beta + dim(X) +m -3$ is often referred to as the \enquote{expected dimension} of the moduli space. One also has evaluation morphisms $$ev_{p_i}: \MM_{0,m}(X,\beta) \to X$$ for $i=1,...,m$, defined as $$ev_{p_i}([(C,p_1,...,p_m,f)]) = f(p_i).$$ Similarly, one has a total evaluation morphism $$ev_m: \MM_{0,m}(X,\beta) \to X^m, \ \ \ ev_m([(C,p_1,...,p_m,f)]) = (f(p_1),...,f(p_m)).$$ The fibers of these evaluation morphisms also have an \enquote{expected dimension}, which equals the expected dimension of the domain minus the dimension of the target.
To prove rational simple connectedness, it will be crucial to understand the general fiber of some of these evaluation morphisms.  This task can be quite hard in general. However, we will only be interested in \enquote{minimal} pointed curves, which have been studied in \cite{dJS06a} and satisfy much nicer properties. This notion is related to the key concept of a \em free \em rational curve, meaning a morphism $f: \PP^1 \to X$ to a smooth variety $X$ such that $f^*(T_X)$ is generated by global sections. The following Theorem collects some of these properties (see \cite{dJS06a}, Lemmas 4.1, 5.1 and 5.3 for further reference):

\begin{thm}
\label{smoothfibers}
Let $X$ be a smooth, projective variety, and $\OO_X(1)$ an ample invertible sheaf. Let $\alpha$ be a curve class of $\OO_X(1)-$degree $1$.

1) Assume that $ev_1: \MM_{0,1}(X,\alpha) \to X$ is dominant. Then a fiber $\mathcal{F}_p$ of $ev_1$ over a general point $p$ is smooth, and every connected component has dimension equal to its expected dimension $-K_X \cdot \alpha -2$.

2) Assume that $ev_2: \MM_{0,2}(X,2\alpha) \to X \times X$ is dominant. Then a fiber $\mathcal{F}_{p,q}$ of $ev_2$ over a general point $(p,q)$ is smooth, and every connected component has dimension equal to its expected dimension $-K_X \cdot (2 \alpha) - dim(X) -1$.
\end{thm}
\begin{proof}
Since the proof of $(1)$ is an easier version of that of $(2)$, we are only going to prove the latter.

Since $p$ and $q$ are general points of $X$ -- in particular, they are distinct and there is no line in $X$ through them -- a point $[(C,p_1,p_2,f)]$ of $\mathcal{F}_{p,q}$ can only parametrize either a stable map of degree $1$ from $\PP^1$ to a conic, or a stable map of degree $1$ from a tree of two copies of $\PP^1$ with one marked point on each component to two lines meeting at one point. In either case, since every component of $C$ has degree $1$ over its image, $(C,p_1,p_2,f)$ is automorphism free. Also, since in either case the image of every component of $C$ passes through a general point of $X$, every component of $(C,p_1,p_2,f)$ is free.

By \cite[Theorem 2]{FP}, this implies that $\MM_{0,2}(X,2\alpha)$ is smooth of the expected dimension at every point of $\mathcal{F}_{p,q}$. Furthermore, if $\mathcal{U} \subset \MM_{0,2}(X,2\alpha)$ is the open subset parametrizing unions of free curves, by Generic Smoothness we have that $ev_2|_{\mathcal{U}}$ is smooth. On the other hand, $\mathcal{F}_{p,q}$ is contained in the fiber of $ev_2|_{\mathcal{U}}$ over $(p,q)$; thus $\mathcal{F}_{p,q}$ is smooth.
\end{proof}

As we mentioned, the properties of Corollary \ref{maincor} will be implied by rational simple connectedness of the weighted complete intersections we are considering. We recall here the definition, which tries to mimic to notion of simple connectedness in topology. Since the spaces of stable maps are (in general) poorly behaved, it is more convenient (and, fortunately, sufficient) to work with suitable \enquote{good} irreducible components of these moduli spaces:

\begin{defin}
\label{RSC}
A variety $X$ is (weakly) \em rationally simply connected \em if, for all $l \geq 2$, there is a canonically defined irreducible component $M_{l,2}$ of $\MM_{0,2}(X,l \alpha)$, with the property that $ev_2|_{M_{l,2}}: M_{l,2} \to X \times X$ is dominant, and the general fiber is rationally connected.
\end{defin}

We will prove these properties in the case $l=2$ directly. It would be hard (in general) to check the definition for infinitely many degrees $l$. However, provided the existence of certain special ruled surfaces, called $1$-twisting surfaces, the proof for $l > 2$ follows by induction (see \cite[Proof of Theorem 1.7]{dJS06a} or \cite[Proof of Theorem 7.3, Step 1]{DeL}). In section $4$, we prove that the general fiber of $ev_2: \MM_{0,2}(X,2\alpha) \to X^2$ is rationally connected (the base case of induction) and that it is uniruled, which allows us to construct these twisting surfaces in section$5$. In section $3$, we prove that the general fiber of $ev_1: \MM_{0,1}(X,\alpha) \to X$ is irreducible; this is the requirement to define the \enquote{good} irreducible components, since we can apply the following Lemma (see \cite[Lemma 3.5]{dJS06a}):

\begin{lem}
\label{higherdegreecomponents}
Let $M_{\alpha,0}$ be an irreducible component of $\MM_{0,0}(X,\alpha)$ whose general point parametrizes a smooth, free curve. Denote by $M_{\alpha,1}$ the unique irreducible component of $\MM_{0,1}(X,\alpha)$ dominating $M_{\alpha,0}$. 

Assume that the general fiber of the restriction $ev|_{M_{\alpha,1}}: M_{\alpha,1} \to X$ is geometrically irreducible. Then for every positive integer $l$ there is a unique irreducible component $M_{l \alpha,0}$ of $\MM_{0,0}(X,l \alpha)$ whose general point parametrizes a smooth, free curve. Denote by $M_{l \alpha,1}$ the unique irreducible component of $\MM_{0,1}(X,l \alpha)$ dominating $M_{l \alpha,0}$. Then the generic fiber of the restriction $ev|_{M_{l \alpha,1}}: M_{l \alpha,1} \to X$ is geometrically irreducible.
\end{lem}

Lemma \ref{higherdegreecomponents} allows us to define $M_{l \alpha,2}$ as the unique irreducible component of $\MM_{0,2}(X,l \alpha)$ dominating $M_{l \alpha,0}$. These components will be the ones of Definition \ref{RSC}. 

\

As a final remark, note that Theorem \ref{mainthm} deals with \em complex \em varieties (and, in fact, we will work over $\CC$ throughout the rest of this paper), whereas our main application (Corollary \ref{maincor}) is about varieties defined over the function field $K$ of a smooth complex curve $B$. With respect to part $(1)$ of Corollary \ref{maincor}, Hassett's result shows that, given a multisection over $B$ passing through prescribed points, rational simple connectedness implies the existence of a ruled surface over $B$ that contains such multisection. One can then prove weak approximation using the fact that it holds for ruled surfaces. With respect to parts $(2)$ and $(3)$ of Corollary \ref{maincor}, since $\overline{K}$ is an algebraically closed field of characteristic 0, by the Lefschetz principle we have an embedding $\overline{K} \subset \CC$. It can be proved that rational simple connectedness over any algebraic closure of $K$ implies a notion of rational simple connectedness over $K$ defined in the same way as Definition \ref{RSC} (\cite[Proposition 4.1]{Pir}). Therefore, once we know that $X_{\CC}$ is rationally simply connected, we also have that $X_K$ is rationally simply connected.

\section{The space of pointed lines}

The main goal of this section is to prove that the evaluation map  $$ev_1: \MM_{0,1}(X,\alpha) \to X$$ is dominant and it has irreducible generic fiber. As we mentioned, we need this to be able to apply Lemma \ref{higherdegreecomponents}, but we will also use this in section $4$ to study the geometry of $2$-pointed stable maps of degree $2$. 

\

Dominance of $ev_1$ follows from \cite[V, 4.11.2]{Kol}, or alternatively from the proof of Theorem \ref{pointedlines} below (in fact, within our degree range, the general fiber of $ev_1$ has even positive dimension). With respect to irreducibility, in \cite{dJS06a} it is proved by means of a geometric argument, that however relies on the properties of lines in projective space. Namely, since through $2$ points there is a unique line, the authors were able to conclude that if there were different components of the fiber, they would have to intersect, and in fact be the same (since the fiber is smooth). However, in the weighted case we might have more than one curve of degree $1$ through $2$ points; therefore proving irreducibility will require some extra work.

\

Since by Theorem \ref{smoothfibers} a general fiber of $ev_1$ is smooth, to prove irreducibility it is enough to prove connectedness. A fundamental ingredient will be the following connectedness result, which can be thought of as a generalization of the Enriques-Severi-Zariski Lemma.

\begin{prop}
\label{ESZ}
Let $X$ be a projective, normal, connected scheme over $\CC$ of dimension at least $2$.  Let $b$ be an integer.  Let $(L_i,s_i,Y_i)_{i=1,...,b}$ be a sequence of triples, where $L_i$ is an ample line bundle on $X$, $Y_0=X$, $s_i \in H^0(Y_{i-1},L_i|_{Y_{i-1}})$, and $Y_i \subset Y_{i-1}$ is the zero locus of $s_i$. Then, if $b$ is less than $dim(X)$, each scheme $Y_1, . . . , Y_b$ is connected.
\end{prop}
\begin{proof}
Observe that, for $b=1$, the statement is the classical Lemma of Enriques-Severi-Zariski. The argument below will prove the Proposition by reducing it to this case.

\

Note first that, since connectedness is not affected by the scheme structure, it suffices to prove the result for $L_i^n$ and $s_i^n$, for some positive integer $n$. This allows us to assume that $L_i$ is very ample, and that $s_i$ is the restriction to $Y_{i-1}$ of a global section on $X$ of $L_i$. Namely, by applying Serre's Vanishing to the $n-$th power of the ideal sheaf of $Y_{i-1}$ (with $n$ big enough), we get that $s_i^n$ is the restriction to $Y_{i-1}$ of a section $t_i$ of $L_i^n$.

\

Now if the schemes $Y_i$ were normal, we could simply apply iteratively the Lemma of Enriques-Severi-Zariski, and we would be done. The trick is to reduce to the case of general sections.

\

Let $V$ denote the vector space $H^0(X,L_1) \oplus ... \oplus H^0(X,L_b)$, so that $\PP(V)$ is the parameter space for $b-$uples $(t_1,...,t_b)$ of global sections $t_i$ of $L_i$ up to common scaling.  Denote by $Z \subset X \times \PP(V)$ the closed subscheme parameterizing pairs $(p,(t_1,...,t_b))$ such that every $t_i$ is zero at $p$. Every $L_i$ is very ample; thus, for every $p \in X$ fixed, the conditions $t_i(p)=0$ for every $i=1,...,b$ correspond to $b$ linearly independent conditions on $\PP(V)$. Therefore the projection $\pi_1: Z \to X$ is a projective bundle. In particular, $Z$ is a projective, connected and normal scheme, and hence irreducible too.

\

We need to show that every fiber of $\pi_2: Z \to \PP(V)$ is connected. Since $Z$ is irreducible, it suffices to prove that the generic fiber of $\pi_2$ is connected.  For $(t_1,...,t_b)$ general, the schemes $Y_i$ are normal by a Bertini type Theorem (\cite[Theorem 7]{Sei}): after embedding $X$ in a projective space via $L_1$, a general hyperplane section is normal by \cite[Theorem 7]{Sei}; therefore $Y_1$ is normal.  By iterating this process (embedding $Y_1$ via $L_2$ and so on) we have that $Y_1,...,Y_b$ are all normal. Therefore the fiber over $(t_1,...,t_b)$ is connected since we can apply the Lemma of Enriques-Severi-Zariski iteratively, starting from $b=1$.
\end{proof}

\

Let us now fix some notation, and the hypothesis required for the main result of this section.

{\bf Notation.} Let $\PP:=\PP(1^3,e_3,...,e_n)$ be a weighted projective space, and let $d_1,...,d_c$ be natural numbers such that $d_1+...+d_c \leq e_3+...+e_n$. Let $W_j:=H^0(\PP,\OO_{\PP}(d_j))$ denote the vector space of degree $d_j$ weighted polynomials, and let $T:=W_1 \oplus ... \oplus W_c$  be the affine space that parametrizes weighted $c-$uples of polynomials of degrees $d_1,...,d_c$. There is a dense open subscheme $\mathcal{U} \subset T$ parametrizing $c-$uples $(F_1,...,F_c)$ whose common zero locus $X_{d_1,..d_c}$ is smooth, contained in $\PP^\circ$ and of codimension $c$. Namely, for $X_{d_1,..d_c}$ to be contained in $\PP^\circ$ corresponds to being disjoint from the closed subscheme $Sing(\PP)$; if $X_{d_1,..d_c}$ is contained in $\PP^\circ$, being smooth is equivalent to being quasismooth, which is an open condition; finally, for $X_{d_1,..d_c}$ to have smallest possible dimension (i.e., codimension $c$) it is also an open condition by upper semicontinuity of fiber dimension. 

\

The following Lemma allows us to restrict to general $c-$uples in $\mathcal{U}$.

\begin{lem}
If the general fiber of $ev_1$ is connected for $X_{d_1,..d_c}$ corresponding to a \em general \em $c-$uple of $\mathcal{U}$, then it is connected for $X_{d_1,..d_c}$ corresponding to \em any \em $c-$uple of $\mathcal{U}$.
\end{lem}
\begin{proof}
Consider the incidence correspondence $$I_1 \subset \PP \times \mathcal{U}$$ consisting of data $(p, (F_1,...,F_c))$ such that $F_j(p)=0$ for $j=1,...,c$ (note that $I_1$ is automatically contained in $\PP^\circ \times \mathcal{U}$). Since $\mathcal{U}$ is smooth and the fibers of the projection $I_1 \to \mathcal{U}$ are the common zero loci of $F_1,...,F_c$ for $(F_1,...,F_c) \in \mathcal{U}$ -- hence smooth -- we have that $I_1$ is smooth.

Consider further the moduli space $\MM_{0,1}(\PP,1)$ of curves $\gamma: \PP^1 \to \PP$ of degree $1$ with $1$ marked point (up to automorphisms), and the incidence correspondence $$I_2 \subset I_1 \times \MM_{0,1}(\PP,1)$$ consisting of data $((p, (F_1,...,F_c)), [\gamma])$ such that: 

1) $\gamma$ maps the marked point to $p$, and

2) $\gamma^*(F_j)=0$ for every $j=1,...,c$ (i.e., the curve is contained in the common zero locus of $F_1,...,F_c$).

(Again, note that $I_2$ is automatically contained in $\PP^\circ \times \mathcal{U} \times \MM_{0,1}(\PP^\circ,1)$.)

\

By construction, the projection $I_2 \to I_1$ is proper.
By upper semicontinuity of the fiber dimension, there is a dense open subscheme $\mathcal{V}_1 \subset I_1$ such that the fibers of the projection $I_2 \to I_1$ over $\mathcal{V}_1$ are equidimensional, with dimension equal to their expected dimension (which is the minimum possibile dimension). Denote by $\mathcal{V}_2 \subset I_2$  the preimage of $\mathcal{V}_1$ under the morphism $I_2 \to I_1$. Then $\mathcal{V}_2 \to \mathcal{V}_1$ is still proper and, by \cite[I.2.17]{Kol}, it is also flat. Therefore, by \cite[Cor. 15.5.4]{EGA}, the number of connected components of the fibers of $\mathcal{V}_2 \to \mathcal{V}_1$ is lower semicontinuous. Since the general fiber of $\mathcal{V}_2 \to \mathcal{V}_1$ is connected by hypothesis, this implies that every fiber is connected.

\

To conclude the proof of the Lemma, we only need to show that the morphism $\mathcal{V}_1 \to \mathcal{U}$ is surjective. Note that there is a dense open subscheme $\mathcal{W}_1 \subset I_1$, contained in $\mathcal{V}_1$, such that the fibers of the projection $I_2 \to I_1$ over $\mathcal{W}_1$ are smooth. Now for \em every \em $(F_1,...,F_c) \in \mathcal{U}$, and for a \em general \em point $p$ in its common zero locus $X_{d_1,..d_c}$, the fiber of $I_2 \to I_1$ over $(p, (F_1,...,F_c))$ is smooth by Theorem \ref{smoothfibers}. This means that the morphism $\mathcal{W}_1 \to \mathcal{U}$ is surjective, and, a fortiori, that the morphism $\mathcal{V}_1 \to \mathcal{U}$ is surjective too.
\end{proof}

We can now prove the main result of this section. The strategy is to consider a birational model of the fiber of the evaluation map, given by a closed subscheme of an appropriate weighted projective space, and then apply Proposition \ref{ESZ}.

\begin{thm}
\label{pointedlines}
Let $X_{d_1,...,d_c}$ be the common zero locus of a general $c-$uple $(F_1,...,F_c) \in \mathcal{U}$. Then the general fiber of $ev_1$ is connected.
\end{thm}
\begin{proof}
Let us denote $X_{d_1,..d_c}$ by $X$ for convenience. Let $q$ denote the point  $[ 1 : 0 : ... : 0 ]$. Since $X$ is general, $X$ intersects the complement of the common zero locus of $x_0, x_1 , x_2$. Therefore, given a general point $x$ of $X$, there exists a change of coordinates of $\PP$ that sends $x$ to $q$.

\

Consider the parameter space $Mor_1(\PP^1,\PP^\circ; \bullet \mapsto q)$ for morphisms of degree $1$ from $\PP^1$ to $\PP^\circ$ that map the marked point to $q$. For every $u \in Mor_1(\PP^1,\PP^\circ; \bullet \mapsto q)$, we can assume -- after composing with an automorphism of $\PP^1$ -- that the marked point mapped to $q$ is $[1:0]$ and that the point $[0:1]$ is the unique point that maps to the hyperplane $\{x_0=0\}$. Then the morphism can be written as $$u([s:t]) =  [s : t u_1 ( s , t ) : ... : t u_n ( s , t )],$$ and the only automorphisms of this stable map are of the form $\phi: [s:t] \mapsto [c_1s:c_2t]$ for some $c_1,c_2 \neq 0$, or equivalently $\phi ([s:t]) = [s:ct]$ for some $c \neq 0$.

Every $u_r(s,t)$ is a homogeneous polynomial of degree $e_r - 1$ of the form $$u_r(s,t) = u_{r,1} s^{e_r-1} + ... + u_{r,e_r} t^{e_r-1},$$ with coefficients $\{u_{r,k}\}$, where $r = 1, ... , n$ and $k = 1 , ... , e_r$.

Composition with the automorphism $\phi([s:t]) = [s:ct]$ above corresponds to multiplying every coefficient $u_{r,k}$ of $u_r(s,t)$ by $c^k$.  Therefore every \em stable map \em (rather than morphism) $u$ as above determines a unique point in a new weighted projective space, $$\mathbb{M} := Proj_\CC [ \{u_{r,k}\} ],$$ where $r \in \{1, ... , n\}$, $k \in \{ 1 , ... , e_r\}$, and $deg(u_{r,k}) = k$.

Thus we can identify the fiber over $q$ of the evaluation morphism $ev_1: \MM_{0,1}(\PP^\circ,\alpha) \to \PP^\circ$ with a dense open subscheme $\mathbb{M}^\circ$ of $\mathbb{M}$.

\

For every weighted polynomial $F$ on $\PP$ that is homogeneous of degree $d$, the polynomial $$(F \circ u)(s,t)=F ( s , t u_1 ( s , t ) , ... , t u_n ( s , t ) )$$ is a polynomial of degree $d$ in $s$ and $t$ of the form $$b_{d,0}^F(\{u_{r,k}\}) s^d + ... + b_{0,d}^F(\{u_{r,k}\}) t^d,$$ with coefficients $b_{d-l,l}^F (\{u_{r,k}\})$ that are polynomials of degree $l$ in the variables $\{u_{r,k}\}$. Note in particular that $b_{d,0}^F$ is constant, and that $F$ vanishes at the point $q$ if and only if $b_{d,0}^F=0$.

\

We can now describe the fiber of the evaluation map as a subset of $\mathbb{M}^\circ$. Let $X$ be defined by polynomials $F_1 , ... , F_c$ on $\PP$ of degrees $d_1 , ... , d_c$. Then the common zero locus of $F_1 , ... , F_c$ is contained in $\PP^\circ$ and it contains $q$. Thus we can describe the fiber of the evaluation map as the intersection of $\mathbb{M}^\circ$ with the common zero locus $\mathbb{M}(F_1,...,F_c)$ on $\mathbb{M}$ of the polynomials $b_{d_j - l , l}^{F_j}$ defined above, for $j = 1 , ... , c$ and for $l = 1 , ... , d_i$.  We can consider these polynomials as sections of ample lines bundles on $\mathbb{M}$; then, by Proposition \ref{ESZ}, $\mathbb{M}(F_1,...,F_c)$ is connected.

\

Since the fiber of the evaluation map is proper, the intersection of $\mathbb{M}(F_1,...,F_c)$ with $\mathbb{M}^\circ$ is proper. This implies that every irreducible component of $\mathbb{M}(F_1,...,F_c)$ is either entirely contained in $\mathbb{M}^\circ$ or it is entirely contained in its closed complement. But since $\mathbb{M}(F_1,...,F_c)$ is connected, $\mathbb{M}(F_1,...,F_c)$ is either entirely contained in $\mathbb{M}^\circ$ or it is entirely contained in its closed complement, which forces $\mathbb{M}(F_1,...,F_c)$ to be contained in $\mathbb{M}^\circ$. Therefore the fiber of the evaluation map equals $\mathbb{M}(F_1,...,F_c)$, which is connected.
\end{proof}

\begin{rmk}
It follows from the proof of Theorem \ref{pointedlines} that $ev_1$ is dominant and that its fibers have positive expected dimension. Namely, $dim(\mathbb{M})=e_1+...+e_n-1$, and every polynomial $F$ of degree $d$ imposes $d$ conditions. Therefore the expected dimension of the fiber is $e_1+...+e_n -1 - (d_1+...+d_c)$, and this number is positive if and only if $d_1+...+d_c \leq e_3+...+e_n$ (which is our hypothesis).
\end{rmk}

\section{The space of 2-pointed conics}

Throughout this section, $X$ will denote a smooth weighted complete intersection $X_{d_1,...,d_c}$ in $\PP_{\CC}(e_0,...,e_n)$ that satisfies the main hypothesis, and $\alpha$ is a curve class of degree $1$. We are going to study the general fiber of $$ev_2: \MM_{0,2}(X,2\alpha) \to X^2.$$ We will show that it is nonempty, irreducible, and uniruled by suitable rational curves. Uniruledness will be used to produce 1-twisting surfaces in section $5$.

\

We will first consider some useful divisor classes on the fiber, and relate them to an ample divisor class $\lambda$. Once again, while in the standard projective case $\lambda$ has some explicit geometric interpretation, that doesn't apply to the weighted case. For this reason, we will define it in a more intrinsic way (which is analogous to the one in \cite{dJS06a} in the case of usual projective space).

\

The setup is the following. Let $M_{p_1,p_2}:=ev_2^{-1}(p_1,p_2)$ be the fiber of $ev_2$ over a point $(p_1,p_2) \in X \times X$, such that $p_1$ and $p_2$ are distinct and not contained in a line in $X$.  Note that in particular -- by inequality $(3)$ in the main hypothesis -- for $(p_1,p_2)$ general there is no line contained in $X$ through $p_1$ and $p_2$. Then we have morphisms:

\begin{center}
\begin{tikzcd}
\mathscr{C} \arrow{r}{g} \arrow{d}{\pi}
&X \subset \PP^\circ \subset \PP \\
M_{p_1,p_2} 
\end{tikzcd}
\end{center}
where $\mathscr{C}$ is the universal family over $M_{p_1,p_2}$, with sections $s_1,s_2:M_{p_1,p_2} \to \mathscr{C}$ corresponding to the marked points. Let $S_i$  be the divisor associated to the zero locus of $s_i$ for $i=1,2$. The line bundles $\OO_{\mathscr{C}}(S_1)$, $\OO_{\mathscr{C}}(S_2)$ and $g^*\OO_{\PP^\circ}(1)$ give divisor classes $c_1(\OO_{\mathscr{C}}(S_1))$, $c_1(\OO_{\mathscr{C}}(S_2))$ and $c_1(g^*\OO_{\PP^\circ}(1))$ in $CH^1(\mathscr{C})$. In $CH^1(M_{p_1,p_2})$, there is a divisor class $\Delta_{1,1}$ corresponding to stable maps with reducible domain.  Furthermore, since $\pi$ is proper, there is a class obtained by pushing forward $c_1(g^*(\OO_{\PP^\circ}(1)))^2 \in CH^2(\mathscr{C})$ via $\pi$. We are now going to define one more class $\lambda \in CH^1(M_{p_1,p_2})$, and establish relations among these classes.

\begin{definlem}
\label{lambdadef}
Assume there is no line through $p_1$ and $p_2$. Then there is a divisor class $\lambda$ on $M_{p_1,p_2}$ that is ample and that satisfies the relation $$\pi_*g^*c_1(\OO_{\PP^\circ}(1))^2=2 \lambda.$$ If furthermore $p_1$ and $p_2$ are general, we have relations $$\Delta_{1,1}=2 \lambda,$$ and $$c_1(T_{M_{p_1,p_2}}) = (\sum e_i^2 -\sum d_j^2 + 2) \lambda.$$
\end{definlem}

\begin{proof}
Let us compare $\OO_{\mathscr{C}}(S_1+S_2)$ and $g^*\OO_{\PP^\circ}(1)$. For a stable map $[(C,p_1,p_2,f)] \in M_{p_1,p_2}$, the domain $C$ is either $\PP^1$ or a union of two copies of $\PP^1$ connected at a node. In the first case, the restriction of $g^*\OO_{\PP^\circ}(1)$ to $(C,p_1,p_2,f)$ is an effective line bundle of degree $2$; in the second case, the restriction of $g^*\OO_{\PP^\circ}(1)$ to $(C,p_1,p_2,f)$ is an effective line bundle obtained by glueing two effective line bundles of degree $1$ on each copy of $\PP^1$. For a stable map $[(C,p,q,f)] \in M_{p_1,p_2}$, the restriction of $\OO_{\mathscr{C}}(S_1+S_2)$ to $(C,p_1,p_2,f)$ has the same description. Therefore $\OO_{\mathscr{C}}(S_1+S_2)$ and $g^*\OO_{\PP^\circ}(1)$ are isomorphic on each irreducible component of each fiber of $\pi$. Thus, by the Semicontinuity Theorem, they are isomorphic up to the pullback of some line bundle $\LL$ on $M_{p_1,p_2}$: $$\OO_{\mathscr{C}}(S_1+S_2) \otimes \pi^* \LL \simeq  g^*\OO_{\PP^\circ}(1).$$

Note that $s_i^*\pi^* \LL \simeq \LL$ since $\pi \circ s_i$ is the identity, and that $s_i^* g^*\OO_{\PP^\circ}(1)$ is trivial since $g \circ s_i$ is the constant map to the point $p_i$. Furthermore, for any divisor $D$ on $\mathscr{C}$, we have that $$s_i^*\OO_{\mathscr{C}}(D) \simeq \pi_* (\OO_{\mathscr{C}}(D)|_{S_i}) \simeq \pi_* \OO_{\mathscr{C}}(S_i \cdot D).$$

Therefore, applying $s_i^*$ to each side of $\OO_{\mathscr{C}}(S_1+S_2) \otimes \pi^* \LL \simeq  g^*\OO_{\PP^\circ}(1)$, we have: $$s_i^* \OO_{\mathscr{C}}(S_i) \simeq \LL^\vee,$$ for $i=1,2$, or equivalently $$\OO_{\mathscr{C}}(\pi_*(S_i \cdot S_i)) \simeq \LL^\vee.$$

We define $\lambda:=c_1(\LL)$. Abusing the notation, $S_i$ will be also used to denote $c_1(\OO_{\mathscr{C}}(S_i))$ for $i=1,2$. Using again the relation $\OO_{\mathscr{C}}(S_1+S_2) \otimes \pi^* \LL \simeq  g^*\OO_{\PP^\circ}(1)$, we have that $$c_1(g^*\OO_{\PP^\circ}(1))^2=(S_1+S_2+\pi^*\lambda)^2=(S_1)^2 +(S_2)^2 +2S_1\cdot \pi^*\lambda +2S_2 \cdot \pi^*\lambda.$$ Since, by the projection formula, $\pi_*(S_i \cdot \pi^* \lambda)=\pi_*(S_i) \cdot \lambda$, and since $\pi_*(S_i) \cdot \lambda= \lambda$ (because $\pi|_{S_i}$ is an isomorphism), we then have:
$$c_1(\pi_*g^*\OO_{\PP^\circ}(1))^2 = \pi_*(S_1)^2 +\pi_*(S_2)^2 +2\pi_*(S_1)\cdot \lambda +2\pi_*(S_2)\cdot \lambda= -\lambda-\lambda+2\lambda+2\lambda=2\lambda.$$ By \cite[Lemma 5.8]{dJS06b}, we have $$\Delta_{1,1}=-\pi_*(S_1)^2 -\pi_*(S_2)^2 =\lambda +\lambda= 2\lambda.$$ The following formula for the canonical class of the fiber is a special case of \cite[Theorem 1.1]{dJS06b} (see \cite[Lemma 6.5]{dJS06a}): $$c_1(T_{M_{p_1,p_2}})=\pi_*g^*(ch_2(T_X)+c_1(\OO_{\PP^\circ}(1))^2).$$ Therefore, combining it with Lemma \ref{cherntwo}, we get $$c_1(T_{M_{p_1,p_2}})=(\frac{1}{2}(\sum e_i^2 -\sum d_j^2) + 1) \cdot 2 \lambda=(\sum e_i^2 -\sum d_j^2 + 2) \lambda.$$

To prove that $\lambda$ is ample, we can proceed as follows. Since $\OO_{X}(1)$ is ample, $\OO_{X}(b)$ is very ample for some positive integer $b$. Let us fix such $b$. Then $\OO_{X}(b)$ determines an embedding of $X$ into some projective space $\PP^N$. By functoriality of Kontsevich spaces, we get an embedding $\MM_{0,2}(X,2\alpha) \hookrightarrow \MM_{0,2}(\PP^N,2\alpha)$. On $\MM_{0,2}(\PP^N,2\alpha)$ there is a divisor class $\mathcal{H}$, corresponding to stable maps whose image intersects a fixed codimension $2$ linear subspace of $\PP^N$. If we write down the morphisms associated to the universal family over $\MM_{0,2}(\PP^N,2\alpha)$:

\begin{center}
\begin{tikzcd}
\mathscr{U} \arrow{r}{G} \arrow{d}{\Pi}
&\PP^N \\
\MM_{0,2}(\PP^N,2\alpha)
\end{tikzcd}
\end{center}
we have that $\mathcal{H}=\Pi_*G^*(c_1(\OO_{\PP^N}(1)))^2$.

It is known that the restriction of $\mathcal{H}$ to $\MM_{0,2}(X,2\alpha)$ is ample away from the locus of multiple covers of lines. Since there is no line contained in $X$ through $p_1$ and $p_2$, the restriction of $\mathcal{H}$ to $M_{p_1,p_2} \subset \MM_{0,2}(X,2\alpha)$ is ample. Furthermore, the restriction of $\mathcal{H}$ to $M_{p_1,p_2}$ has divisor class $\pi_*g^*c_1(\OO_{X}(b))^2$, which is equal to $2b^2\lambda$. Therefore $\lambda$ is ample.
\end{proof}

\

Now we can use $\lambda$ as an aid to study the geometry of fibers of $ev_2$. Before that, we need to relate $\lambda$ to the tangency divisor $\mathcal{T}_{W,p_1}$, which we now define.

\

Recall the setup of Proposition \ref{extend}: let $Y \subset \PP(e_0,..,e_n,1^s)$ be a smooth weighted complete intersection that extends the smooth weighted complete intersection $X \subset \PP(e_0,...,e_n)$, meaning that $X=Y \cap \{y_1=...=y_s=0\}$, where $y_1,...,y_s$ are the extra coordinates of weight $1$. Let $M_Y$ be the fiber of $ev_2: \MM_{0,2}(Y,2\alpha) \to Y^2$ over a general point $(p_1,p_2) \in X^2$. Let $W \subset \PP(e_0,..,e_n,1^s)$ be the zero locus of one such coordinate $y_i$ for some fixed $i=1,...,s$. Then $W$ passes through $p_1$, and $W \cap Y$ is a smooth codimension $1$ closed subscheme of $Y$, so that $T_{p_1}(Y \cap W) \subset T_{p_1}Y$ is a codimension $1$ linear subspace.  We define $\mathcal{T}_{W,p_1}$ as the Cartier divisor in $M_Y$ corresponding to stable maps $\{[f: C \to Y;t_1,t_2 \in C]\ |\ f(t_1)=p_1,\ f(t_2)=p_2,\ im(df_{t_1}) \subset T_{p_1}(W \cap Y)\}$. Note that $$s_1^*(\Omega_\pi)|_{[f: C \to Y;t_1,t_2 \in C]} \simeq Hom(T_{t_1}C, \CC),$$ and there is a global section $dg_{s_1}$ of $s_1^*(\Omega_\pi)$ defined over a point $[f: C \to Y;t_1,t_2 \in C, f(t_1)=p_1, f(t_2)=p_2]$ by the morphism $$T_{t_1}C \xrightarrow{df_{t_1}} T_{p_1}Y \to \frac{T_{p_1}Y}{T_{p_1}(W \cap Y)} \simeq \CC.$$ Therefore the divisor $\mathcal{T}_{W,p_1}$ can be thought of as the divisor of zeroes of $dg_{s_1}$. Since $s_1$ is a section, its image $S_1$ is contained in the smooth locus $\mathscr{C}^{sm}$ of $\mathscr{C}$. Therefore, the relative cotangent sequence associated to $S_1 \hookrightarrow \mathscr{C}^{sm} \to M_{p_1,p_2}$ is the short exact sequence: $$0 \to \OO_{\mathscr{C}}(-S_1)|_{S_1} \to \Omega_{\pi}|_{S_1} \to \Omega_{S_1/M_Y} \to 0.$$ Note that $\OO_{\mathscr{C}}(-S_1)|_{S_1} \simeq s_1^*(\OO_{\mathscr{C}}(-S_1))$ and $ \Omega_{\pi}|_{S_1} \simeq s_1^*\Omega_\pi$. Since $S_1 \to M_Y$ is an isomorphism, $\Omega_{S_1/M_Y} = 0$. Therefore we have that the divisor class of $\mathcal{T}_{W,p_1}$ is $c_1(s_1^*\Omega_\pi)=-\pi_*(S_1 \cdot S_1)$.

Now note that since $(p_1,p_2)$ is a general point of $X^2$, there is no line in $X$ that contains $p_1$ and $p_2$. This also implies that there is no line in $Y$ passing through $p_1$ and $p_2$; namely since each coordinate $y_i$ vanishes at $p_1$ and $p_2$, if there was such a line it would be contained in the common zero locus $\{y_1=...=y_s=0\}$, and therefore in $X$. The fact that there is no line in $Y$ that contains $p_1$ and $p_2$ is enough to define the divisor class $\lambda$ on $M_Y$ as in Definition-Lemma \ref{lambdadef}, which by construction equals $-\pi_*(S_1 \cdot S_1)$, and to prove it is ample. Therefore, tying the ends up together, the class of $\mathcal{T}_{W,p_1}$ equals $\lambda$, and it is ample.

\

We can now prove that a general fiber of $ev_2$ is non-empty and irreducible.

\begin{prop}
\label{irreducibleconics}
The general fiber of $ev_2: \MM_{0,2}(X,2\alpha) \to X^2$ is non-empty and irreducible.
\end{prop}

\begin{proof}

By Proposition \ref{extend}, we can extend $X$ to some $Y \subset  \PP(e_0,..,e_n,1,...,1)$ such that $X$ is contained in $Y$, $X$ equals $Y \cap D_1 \cap ... \cap D_r$, $D_k$ Cartier divisor of degree $1$ for every $k=1,...,r$. By adding enough variables, i.e., by taking $r$ big enough, we can choose $Y$ that satisfies the inequality $$dim(Y) \leq -K_Y \cdot 2\alpha -2.$$ In fact, $dim(Y)=dim(X)+r$ and $ -K_Y \cdot 2\alpha -2=  -K_X \cdot 2\alpha -2 +2r$. Similarly, we can also assume that $$2 \cdot (-K_Y \cdot \alpha -2) > n+r=dim(\PP(e_0,..,e_n,1,...,1)).$$ We will make use of the first inequality right away, and of the second one later on.

\

Let us consider the fiber $M_{p_1,p_2}(Y)$ of $ev_2(Y): \MM_{0,2}(Y,2\alpha) \to Y^2$ over a general point $(p_1,p_2) \in Y^2$. We first show that it is non-empty. For $i=1,2$, $ev_{p_i}: \MM_{0,p_i}(Y,\alpha) \to Y$ is dominant, and the general fiber is connected of dimension $-K_Y \cdot \alpha -2$. The curves parametrized by a general fiber of $ev_{p_i}$ sweep out a closed subscheme $\Pi_i$ of $Y$ of dimension $(-K_Y \cdot \alpha -2) +1=-K_Y \cdot \alpha -1$. For $i=1,2$, $dim(\Pi_i)=-K_Y \cdot \alpha -1$, so that $dim(\Pi_1) + dim(\Pi_2)=2(-K_Y \cdot \alpha -1) \geq dim(Y)$. Therefore, by Theorem \ref{cohomology}, $\Pi_1 \cap \Pi_2$ is nonempty, which means that there are two curves $C_1$ and $C_2$ of class $\alpha$ containing $p_1$ and $p_2$ respectively, that intersect at a point. Thus $C_1 \cup C_2$ has class $2\alpha$ and contains $p_1$ and $p_2$, which implies that $ev_2(Y)$ is dominant, and therefore surjective.

\

The inequality $dim(Y) \leq -K_Y \cdot 2\alpha -2$ also allows us to use Bend and Break (\cite[Corollary 5.6.2]{Kol}): every component of $M_{p_1,p_2}(Y)$ contains a reducible curve, and thus every component of $M_{p_1,p_2}(Y)$ intersects the boundary divisor $\Delta_{1,1}(Y) \subset \MM_{0,2}(Y,2)$. Thus to prove connectedness, it's enough to show that $\Delta_{1,1}(Y) \cap M_{p_1,p_2}(Y)$ is connected. To do this, we interpret $\Delta_{1,1}(Y) \cap M_{p_1,p_2}(Y)$ as the preimage of a diagonal and use B\u adescu's theorem (\cite{Bad}). 

\

Consider the evaluation map $e_2: \MM_{0,2}(Y,\alpha) \to Y \times Y$. In what follows, we denote the diagonal in $Y \times Y$ as $\Delta_Y$.
Define $M_{p_1,\bullet} :=e_2^{-1}(\{p_1\} \times Y)$, $M_{\bullet, p_2} :={e_2}^{-1}(Y \times \{p_2\})$. $M_{p_1,\bullet}$ is projective, smooth and connected. This follows from the fact that a general fiber of $ev_1: \MM_{0,1}(Y,\alpha) \to Y$ is smooth and connected, and uses the following commutative diagram:

\begin{center}
\begin{tikzcd}
\MM_{0,2}(Y,\alpha) \arrow{r}{e_2} \arrow{d}
&Y \times Y \arrow{d}{\pi_1}\\
\MM_{0,1}(Y,\alpha) \arrow{r}{ev_1} &Y
\end{tikzcd}
\end{center}

(Here the map on the lefthand side forgets the second marked point.)

Similarly, $M_{\bullet, p_2}$ is projective, smooth and connected.

\

We have maps $M_{p_1,\bullet} \to Y$ and $M_{\bullet, p_2} \to Y$, given by evaluation at the unspecified marked point $\bullet$. Their product defines a morphism $$e_{1,2}: M_{p_1,\bullet} \times M_{\bullet, p_2} \to Y \times Y.$$ Since through any two distinct points there are at most finitely many lines, $e_{1,2}$ is finite over its image. We also have $e_{1,2}^{-1}(\Delta_Y)=\Delta_{1,1}(Y) \cap M_{p_1,p_2}(Y)$.

Since $M_{p_1,\bullet}$ and $M_{\bullet, p_2}$ are irreducible, $M_{p_1,\bullet} \times M_{\bullet, p_2}$ is irreducible of dimension $dim(M_{p_1,\bullet}) + dim(M_{\bullet, p_2})=2(-K_X \cdot \alpha -2)$. Therefore $M_{p_1,\bullet} \times M_{\bullet, p_2}$ is $(dim(M_{p_1,\bullet} \times M_{\bullet, p_2})-1)-$connected, and by the second inequality at the beginning of the argument, $$dim(M_{p_1,\bullet} \times M_{\bullet, p_2}) > dim(\PP(e_0,..,e_n,1,...,1)).$$ Therefore B\u adescu's theorem (\cite{Bad}) implies that  $e_{1,2}^{-1}(\Delta_Y)=\Delta_{1,1}(Y) \cap M_{p_1,p_2}(Y)$ is connected, which finally gives that $M_{p_1,p_2}(Y)$ is connected.

\

Now, by Theorem \ref{smoothfibers}, a general fiber of $ev_2(Y)$ is smooth (of expected dimension). Therefore a general fiber of $ev_2(Y)$ is irreducible.

\

To conclude, consider the fiber $M_{p_1,p_2}(X)$ of $ev_2(X): \MM_{0,2}(X,2\alpha) \to X^2$ over a general point $(p_1,p_2) \in X \times X$. Let $M_{p_1,p_2}(Y)$ denote the fiber over $(p_1,p_2)$ of $ev_2(Y)$. We are going to show that $M_{p_1,p_2}(Y)$ is connected and normal.

Notice that $M_{p_1,p_2}(X)$ is cut out in $M_{p_1,p_2}(Y)$ by the ample divisors $\mathcal{T}_{Y \cap D_k,p_1}$, for $k=1,...,r$. Therefore $dim(M_{p_1,p_2}(Y)) \leq dim(M_{p_1,p_2}(X)) + r$. In particular, since $M_{p_1,p_2}(Y)$ is non-empty, $M_{p_1,p_2}(X)$ is also non-empty.

Since $(p_1,p_2) \in X \times X$ is general, we also have that $$dim(M_{p_1,p_2}(X)) + r= -K_X \cdot 2\alpha -1 + r=-K_Y \cdot 2\alpha -1 \leq$$ $$\leq dim(\MM_{0,2}(Y,2\alpha)) - 2 dim(Y) \leq dim(M_{p_1,p_2}(Y)).$$

Therefore the dimension of $M_{p_1,p_2}(Y)$ equals its expected dimension.

Since $(p_1,p_2) \in X \times X$ is general, by Theorem \ref{smoothfibers} we have that $M_{p_1,p_2}(X)$ is smooth. Since $M_{p_1,p_2}(X)$ is cut by ample divisors in $M_{p_1,p_2}(Y)$, it follows that $M_{p_1,p_2}(Y)$ is regular in codimension $1$.

Now consider an integral (say, rational) curve T in $Y \times Y$ connecting $(p_1,p_2)$ to a general point $(q_1,q_2) \in Y \times Y$. We proved above that the dimension of the fiber of $ev_2(Y)$ over $(p_1,p_2)$ equals its lower bound. Therefore there is an open dense subset $T^\circ \subset T$ containing $(p_1,p_2)$ such that the dimension of the fibers of $ev_{T} : ev_2(Y)^{-1}(T^\circ) \to T^\circ$ is the smallest possible. By \cite[II, 1.7.3]{Kol}, $ev_{T}$ is a flat, local complete intersection morphism. Since we showed in the first part of the proof that a general fiber of $ev_{T}$ is connected, all fibers of $ev_{T}$ are connected by the Principle of Connectedness (\cite[III, Exercise 11.4]{Har}). In particular, $M_{p_1,p_2}(Y)$ is connected. Since $M_{p_1,p_2}(Y)$ is a local complete intersection and regular in codimension $1$, it is also normal.

\

As already noticed above, $M_{p_1,p_2}(X)$ is cut out in $M_{p_1,p_2}(Y)$ by the ample divisors $\mathcal{T}_{Y \cap D_1,p_1}$, ..., $\mathcal{T}_{Y \cap D_r,p_1}$. Thus, by Lemma \ref{ESZ}, $M_{p_1,p_2}(X)$ is connected. By Theorem \ref{smoothfibers}, a general fiber of $ev_2(X)$ is smooth. Therefore a general fiber of $ev_2(X)$ is irreducible.

\end{proof}

\

We can finally prove uniruledness of a general fiber of $ev_2$, which we will use to construct $1$-twisting surfaces.

\begin{prop}
Assume that $\sum_{i=0}^n e_i^2 -\sum_{j=1}^c d_j^2 \geq 0$. Then a general fiber of $ev_2$ is uniruled by rational curves of $\lambda-$degree $1$.
\end{prop}

\begin{proof}
First of all, we observe that if $2(\sum e_i^2 -\sum d_j^2) + dim(X)+2K_X \cdot \alpha+4 > 0$, then a general fiber $M_X$ of $ev_2$ is uniruled by rational curves of $\lambda-$degree $1$. Namely, since $c_1(T_{M_X}) = (\sum e_i^2 -\sum d_j^2 + 2) \lambda$ and $\lambda$ is ample, $M_X$ is uniruled by curves of $(-K_{M_X})-$degree at most $dim(M_X)+1$. Let $\gamma$ be such a curve of minimal degree. If we had $\lambda \cdot \gamma \geq 2$, then $2(\sum e_i^2 -\sum d_j^2) + dim(X)+2K_X \cdot \alpha+4 > 0$ would imply that $-K_{M_X} \cdot \gamma > dim(M_X)+1$, a contradiction.

\

Assuming that the general fiber of $ev_2$ is uniruled by rational curves of $\lambda-$degree $1$, and that $D$ is such a curve, the dimension of the space of deformations of $D$ containing a general point is $-K_{M_X} \cdot D -2 =\sum e_i^2 -\sum d_j^2.$ The next step is to show that the inequality $-K_{M_X} \cdot D -2 =\sum e_i^2 -\sum d_j^2 \geq 0$ is actually sufficient to guarantee that the general fiber of $ev_2$ is uniruled by rational curves of $\lambda-$degree $1$.

\

By Proposition \ref{extend}, we can extend $X$ to some $Y \subset  \PP(e_0,..,e_n,1,...,1)$ such that $X \subset Y$, $X=Y \cap D_1 \cap ... \cap D_r$, $D_k$ Cartier divisor of degree $1$ for every $k=1,...,r$. By adding enough variables, we can choose $Y$ that satisfies the inequality $2(\sum e_i^2 +r -\sum d_j^2) + dim(Y)+2K_Y \cdot \alpha-4 > 0$.

\

Let $(p_1,p_2) \in X^2$ be a general point, and let $M_X, M_Y$ be the fibers over $(p_1,p_2)$ of $ev_2(X): \MM_{0,2}(X,2) \to X^2$ and $ev_2(Y): \MM_{0,2}(Y,2) \to Y^2$ respectively.

By \cite[2.10]{Deb}, the fiber of the $Y \times Y-$scheme $Mor(\PP^1\times Y \times Y, \MM_{0,2}(Y,2))$ over a point $(q_1,q_2) \in Y \times Y$ is $Mor(\PP^1, ev_2(Y)^{-1}(q_1,q_2))$. Therefore, by upper semicontinuity of the fiber dimension, the dimension of the space of rational curves in $ev_2(Y)^{-1}(q_1,q_2)$ of $\lambda-$degree $1$ containing a general point is at least $-K_Y \cdot D -2$ for \em any \em $(q_1,q_2) \in Y \times Y$. In particular, the dimension of the space of rational curves of $\lambda-$degree $1$ in $M_Y$ passing through a general point is at least $-K_Y \cdot D -2$.

Consider a general point $[C]$ of $M_X$ given by a stable map with image $C$. Then $C$ is smooth and very free, which means that $C \simeq \PP^1$ and that $h^1(C, T_X|_C(-2))=0$. The normal exact sequence of $X$ in $Y$ restricted to $C$ is: $$0 \to T_X|_C \to T_Y|_C \to \OO_X(1)^{\oplus c}|_{C} \to 0.$$ Since $\OO_X(1)|_C \simeq \OO_C(2)$, after twisting the short exact sequence by $\OO_C(-2)$ we get $$0 \to T_X|_C(-2) \to T_Y|_C(-2) \to \OO_C^{\oplus c} \to 0.$$ Since $h^1(C, T_X|_C(-2))=0$ and $h^1(C,\OO_C)=0$, we have $h^1(C, T_Y|_C(-2))=0$, which means that $[C]$ is also very free \em as a point of $M_Y$\em. Therefore the space $\mathscr{M}(M_Y,[C])$ of $\lambda-$degree $1$ curves in $M_Y$ passing through $[C]$ has dimension at least $-K_Y \cdot D -2=r -K_X \cdot D -2$.

Let $\mathscr{M}(M_Y,[C])^{norm}$ denote the normalization of $\mathscr{M}(M_Y,[C])$. By \cite[Theorem 3.4]{Keb}, there is a \enquote{tangency} morphism $\tau: \mathscr{M}(M_Y,[C])^{norm} \to \PP(T_{M_Y,[C]})$ which is finite. The morphism $\tau$ associates to a $\lambda-$degree $1$ curve in $M_Y$, smooth at $[C]$, its tangent direction at $[C]$. In the normalization $\mathscr{M}(M_Y,[C])^{norm}$, a point parametrizing a curve singular at $[C]$ is replaced by points parametrizing the normalization of such curve, and the tangency map gives the various tangency lines. Since $\tau$ is finite, its image $Im(\tau)$ is a projective scheme of dimension at least $-K_Y \cdot D -2=r -K_X \cdot D -2$. Since $M_X=M_Y \cap \mathcal{T}_{D_1,p_1} \cap ... \cap \mathcal{T}_{D_r,p_1}$, where $\mathcal{T}_{D_1,p_1}, ..., \mathcal{T}_{D_r,p_1}$ are the tangency divisors defined earlier in the section, we have that $\PP(T_{M_X,[C]}) \subset \PP(T_{M_Y,[C]})$ is a linear subspace of codimension $r$. Since, by hypothesis, $-K_X \cdot D -2 \geq 0$, we get that $Im(\tau)$ and $\PP(T_{M_X,[C]})$ have nonempty intersection. This means that there exists a $\lambda-$degree $1$ curve $\gamma$ in $M_Y$ through $[C]$ with a tangent line at $[C]$ that belongs to $T_{M_X,[C]}$, and therefore to every tangent space $T_{\mathcal{T}_{D_k,p_1},[C]}$ for $k=1,...,r$. Since $\mathcal{T}_{D_k,p_1}$ has class $\lambda$ for every $k=1,...,r$, the curve $\gamma$ must lie entirely in $\mathcal{T}_{D_k,p_1}$ for every $k=1,...,r$, i.e., $\gamma$ must lie in $M_X$.
\end{proof}

\section{Twisting surfaces and Rational Simple Connectedness}

We are now finally able to produce our twisting surface, following the strategy of \cite[Lemma 6.8]{dJS06a}. Recall the definition of a $1$-twisting surface (in a version sufficient for our purposes):

\begin{defin}
A \em $1$-twisting surface \em in a scheme $X$ is a surface $\Sigma \simeq \PP^1 \times \PP^1$ with first projection $\pi: \Sigma \to \PP^1$ and with a morphism $h: \Sigma \to X$ such that:

1) the morphism $(\pi, h): \Sigma \to \PP^1 \times X$ is finite;

1) the vertical normal bundle $N_{(\pi, h)}:=$Coker$(d(\pi, h): T_{\Sigma} \to (\pi, h)^*T_{\PP^1 \times X})$ is generated by global sections; and

2) $h^1(\Sigma, N_{(\pi,h)}(-1,-1))=0$.
\end{defin}

In the next Proposition we will construct a candidate $\Sigma$ for our twisting surface, a candidate which we will be able to conclude it is actually $1$-twisting by applying a modified version of \cite[Lemma 7.8]{dJS06a}.

\begin{prop}
\label{onetwisting}
Let $X$ be a smooth weighted complete intersection satisfying the \em main hypothesis\em. Then a general point of a general fiber of $ev_2: \MM_{0,2}(X, 2 \alpha) \to X^2$ parametrizes a divisor on a surface $\Sigma \simeq \PP^1 \times \PP^1$ with class $(1,1)$. Furthermore, there is a morphism $h: \Sigma \to X$ such that the morphism $(\pi, h): \Sigma \to \PP^1 \times X$ is finite.
\end{prop}

\begin{proof}

Let $D \subset M_{p_1,p_2}$ be a free rational $\lambda-$degree $1$ curve that meets the boundary divisor $\Delta_{1,1} \subset M_{p_1,p_2}$ transversally at two points. Consider the diagram:

\begin{center}
\begin{tikzcd}
\mathscr{C}_D \arrow[hook]{r} \arrow{d}{\pi |_D}
&\mathscr{C} \arrow{r}{g}\arrow{d}{\pi}
&X \subset \PP^\circ \subset \PP\\
D\arrow[hook]{r}& M_{p_1,p_2}
\end{tikzcd}
\end{center}
where $\mathscr{C}_D$ is the restriction of $\mathscr{C}$ over $D$. Then $\mathscr{C}_D$ is a smooth rational surface, fibered via $\pi$ over $D \simeq \PP^1$, with two sections $R_1$ and $R_2$. If $S_1$ and $S_2$ are the two sections of $\pi: \mathscr{C} \to M_{p_1,p_2}$, we have seen that $\pi_*(S_1)^2=\pi_*(S_2)^2=-\lambda$. Therefore $R_1$ and $R_2$ have self-intersection $-1$, and thus can be contracted by Castelnuovo's Theorem. Denote by $\phi: \mathscr{C}_D \to \Sigma$ the contraction of $R_1$ and $R_2$. 

\

The fibration $\pi: \mathscr{C}_D \to D$ has exactly two reducible fibers over, say, $m_1$ and $m_2$. In particular, $\pi^{-1}(m_i)=F_i \cup G_i$, $F_i \simeq \PP^1 \simeq G_i$, for $i=1,2$. Say also that $F_1, F_2$ intersect $R_1$ but not $R_2$, and that $G_1, G_2$ intersect $R_2$ but not $R_1$. Since $F_i$ and $G_i$, for fixed $i$, are irreducible components of connected rational fibers, we have that $F_i^2=G_i^2=-1$ for $i=1,2$. In particular, by Castelnuovo's Theorem, we can contract $F_1$ and $G_2$. Let $\phi': \mathscr{C}_D \to \Sigma'$ be such contraction. Then $\Sigma'$ is a $\PP^1-$bundle, and it is easy to check that $(\phi'(R_i))^2=0$ for $i=1,2$. Thus, by \cite[V, Theorem 2.7a]{Har}, we have that $\Sigma' \simeq \PP^1 \times \PP^1$. We can conclude that $\Sigma$ is a del Pezzo surface with $K_{\Sigma}^2= K_{\Sigma'}^2=8$ and at least four $(-1)-$curves. By the classification of del Pezzo surfaces (see for example \cite[III, 3.9]{Kol}), it follows that $\Sigma \simeq \PP^1 \times \PP^1$.

\

Now consider the two morphisms:

\[
\begin{tikzcd}
 & \Sigma \\
\mathscr{C}_D \arrow{ur}{\phi} \arrow{rr}{g|_{\mathscr{C}_D}} && X
\end{tikzcd}
\]
Since the fibers of $\phi$ are contracted by $g|_{\mathscr{C}_D}$, the latter map factors through the former by a rigidity result (see \cite[Lemma 1.15]{Deb}). Therefore we obtain a morphism $h: \Sigma \to X$. Furthermore, the morphism $(\pi, h): \Sigma \to \PP^1 \times X$ is finite by construction since, for $t \in \PP^1 \simeq D$, it reduces to the stable map $g_t: \mathscr{C}_t \to X$.

\end{proof}

Since $h$ is equal to $g|_{\mathscr{C}_D}$ on a dense open subscheme of $\Sigma$, and $\mathscr{C}_D$ is a family of free curves, we have that $h$ is unramified on a dense open subscheme of $\Sigma$. Therefore (see \cite[Definition 3.4.5]{Ser}) we have an injection $dh: T_{\Sigma} \to h^*T_X$, whose cokernel is denoted $N_h$. In this setting, $H^0(\Sigma, N_h)$ describes first order deformations of $h: \Sigma \to X$ keeping the target fixed, and $H^1(\Sigma, N_h)$ describes obstructions to infinitesimal deformations.

\

Since $\Sigma$ is abstractly isomorphic to $\PP^1 \times \PP^1$, we can deduce that $h: \Sigma \to X$ is $1-$twisting as in \cite[Lemma 7.8]{dJS06a}. The only difference is that in \cite{dJS06a} the morphism $h$ turns out to be an embedding, and one can deduce properties for $N_{(\pi, h)}$ by studying the normal bundle $N_{\Sigma/X}$. In our case, the normal bundle has to be replaced by the normal sheaf $N_h$ associated to $h$, but the result still applies.

\

The existence of a $1$-twisting surface was the last ingredient needed to deduce rational simple connectedness from the general machinery developed in \cite{dJS06a}. We now sketch this process for the reader's convenience.

\

$\bullet$ Since the general fiber of $ev_1: \MM_{0,1}(X,\alpha) \to X$ is connected by Theorem \ref{pointedlines}, by Lemma \ref{higherdegreecomponents} there are canonically defined irreducible components $M_{l \alpha,0}$ of $\MM_{0,0}(X,l \alpha)$ for every $l \geq 1$. A general point of $M_{l \alpha,0}$ parametrizes a smooth, free curve of degree $l$; furthermore, $M_{l \alpha,0}$ also parametrizes degree $l$ covers of smooth, free curves, as well as reducible curves whose components are smooth and free.

\

$\bullet$ For $l,m \geq 1$, we define $M_{l \alpha,m} \subset \MM_{0,m}(X,l \alpha)$ as the irreducible component whose image in $\MM_{0,0}(X,l \alpha)$ under the forgetful morphism is $M_{l \alpha,0}$. For $m=2$, these are the components that appear in Definition \ref{RSC}.

\

$\bullet$ Consider a surface $\Sigma \simeq \PP^1 \times \PP^1$ with first projection $\pi: \Sigma \simeq \PP^1$ and with a morphism $h: \Sigma \to X$. Then $h: \Sigma \to X$ is said to have $M$-class $(l_1,l_2)$ if the fibers of $\pi$ are parametrized by points of $M_{l_1 \alpha,0}$, and the sections of $\pi$ by a point of $M_{l_2 \alpha,0}$.

\

$\bullet$ In Proposition \ref{onetwisting}, we constructed a $1$-twisting surface with $M$-class $(1,1)$. By \cite[Section 7]{dJS06a}, there are $1$-twisting surfaces with $M$-class $(l_1,l_2)$ for every $l_1,l_2 \geq 1$.

\

$\bullet$ Finally, to show rational simple connectedness we need to show that the general fiber of $ev_2|_{M_{l\alpha,2}}: M_{l \alpha,2} \to X \times X$ is rationally connected for every $l \geq 2$. For $l=2$, we know that the general fiber is irreducible from Proposition \ref{irreducibleconics}, and that it is Fano (hence rationally connected) from Definition-Lemma \ref{lambdadef}. The result for $l > 2$ follows by induction, using the existence of $1$-twisting surfaces of $M$-class $(1,l)$, in a rather sophisticated way (see \cite[Proof of Theorem 1.7]{dJS06a} or \cite[Proof of Theorem 7.3, Step 1]{DeL}). We summarize the argument in the remaining part of the section.

\

The core difficulty is in proving the following Claim, which requires the existence of a $1$-twisting surface.

{\bf Claim:} A general point of a general fiber of $ev: M_{l\alpha,2} \to X \times X$ is contained in a rational curve intersecting the boundary $\Delta_{1,l-1}$ in a smooth point of the fiber.

\begin{proof}
Let $h: \Sigma \to X$ be a $1$-twisting surface of type $(1,l-1)$. The map $h$ induces a morphism $\MM_{0,2}(\Sigma,F'+F) \to \MM_{0,2}(X,l)$, where $F$ is the fiber of the first projection $\pi$ and $F'$ is the image of a section of $\pi$. Curves corresponding to divisors $D \in |\OO_\Sigma(F'+F)|$ with $2$ distinct smooth marked points are smooth points of $\MM_{0,2}(X,l)$. One can deform $F+F'$ to a smooth divisor $C$ in the linear system while keeping the marked points fixed. The two divisors $F+F'$ and $C$ span a pencil $\Lambda \simeq \PP^1$, which induces a morphism $\PP^1 \to M_{l\alpha,2}$ whose image is contained in the fiber of the evaluation map. One can deform the stable map $f_C$ corresponding to $C$ in $M_{l\alpha,2}$ to a general stable map $f'$ of the fiber. Since $\Sigma$ is $1$-twisting, $\Sigma$ can be deformed to a $1$-twisting surface $\Sigma'$ that contains a deformation $C'$ of $C$ whose induced stable map is $f'$.
\end{proof}

Using the Claim as a starting point, we can reduce the general statement to the following one: 

{\bf Reduction:} If the general fiber of $ev_\Delta: M_{\alpha,2} \times_{X} M_{(l-1)\alpha,2}$ is rationally connected, then the general fiber of $ev_2|_{M_{l\alpha,2}}: M_{l \alpha,2} \to X \times X$ is rationally connected.

The reduction is a consequence of the following more general Lemma (in our case, $V$ is the general fiber of $ev_\Delta$ and $W$ is the general fiber of $ev_2|_{M_{l\alpha,2}}$):

{\bf Lemma:} Let $V \subset W$ be projective varieties such that $V$ is connected and rationally connected and it intersects the smooth locus $W^{sm}$ of $W$. Assume further that for a general point $w \in W$, there is an irreducible rational curve in $W$ that intersects $V$ in a point that is a smooth point of $W$. Then $W$ is also rationally connected.

\begin{proof}
Let $\nu: \widetilde{W} \to W$ be a resolution of singularities of $W$, and let $\widetilde{V} \subset \widetilde{W}$ be the proper transform of $V$ in $\widetilde{W}$. By \cite[Theorem 5.4]{Kol}, there exists a rational map (the MRC fibration) $\phi: \widetilde{W} \dashedrightarrow Q$ to the MRC quotient $Q$. This means that there is a dense open $U \subset \widetilde{W}$ such that the restriction $\phi_U: U \to Q$ of $\phi$ is a proper morphism, a (non-empty) general fiber of $\phi_U$ is rationally connected, and a (non-empty) very general fiber of $\phi_U$ is a rationally connected component. By hypothesis, and since $\nu$ is an isomorphism on $W^{sm}$, a general point of $\widetilde{W}$ is contained in a rational curve that intersects $\widetilde{V}$. Since $\phi$ is the MRC fibration, this implies that $\widetilde{V}$ intersects the generic fiber of $\phi_U$, which means that $\widetilde{V} \cap U$ is non-empty and that $\phi_U(\widetilde{V} \cap U) \subset Q$ is dense. Since $V$ is connected and rationally connected, $\widetilde{V}$ is also connected and rationally connected, and since $\phi_U(\widetilde{V} \cap U) \subset Q$ is dense, $Q$ is connected and rationally connected too. Therefore $\widetilde{W}$ is rationally connected, and hence $W$ is rationally connected.
\end{proof}

\

Because of the reduction, we only need to show that the general fiber of $ev_\Delta: M_{\alpha,2} \times_{X} M_{(l-1)\alpha,2}$ is rationally connected. The sketch of the argument is the following, and it uses the \em induction hypothesis \em that the general fiber of $ev: M_{(l-1)\alpha,2} \to X \times X$ is rationally connected. For $(x_1,x_2) \in X \times X$ general point, let $F$ be the fiber over $x_1$ of the morphism $M_{\alpha,2} \to X$, which can be written equivalently as $M_{\alpha,2} \xrightarrow{ev_2} X\times X \xrightarrow{pr_1} X$ and as $M_{\alpha,2} \xrightarrow{\Phi_1} M_{\alpha,1} \xrightarrow{ev_1} X$ (where $\Phi_1$ forgets the first marked point). By looking at the first factorization, we can deduce that $ev_\Delta^{-1}(x_1,x_2)$ is mapped onto $F$ by the projection $M_{\alpha,2} \times_{X} M_{(l-1)\alpha,2} \xrightarrow{\pi_1} M_{\alpha,2}$. By looking at the second factorization, it also follows that $F$ is rationally connected. Since $F$ is rationally connected and the general fiber over $F$ of $\pi_1$ is rationally connected by induction (since it equals the general fiber of $ev_2:M_{(l-1)\alpha,2} \to X\times X$), by \cite{GHS} we have that $ev_\Delta^{-1}(x_1,x_2)$ is also rationally connected.

\bibliography{weak_app}

\end{document}